\documentclass{amsart}
\usepackage[usenames]{color} 
\usepackage{amssymb} 
\usepackage{amsmath} 
\usepackage{amsthm}
\usepackage{tikz}

\usepackage[colorlinks, linkcolor=blue,  citecolor=blue, urlcolor=blue]{hyperref}%

\usepackage{url}
\usepackage{amssymb}
\usepackage{verbatim}

\usepackage{alltt}
\usepackage{graphicx}
\usepackage{float}
\usepackage{subcaption}

\graphicspath{{figures/}}

\newtheorem{theorem}{Theorem}[section]

\newtheorem{lemma}[theorem]{Lemma}
\newtheorem{proposition}[theorem]{Proposition}
\newtheorem{problem}[theorem]{Problem}

\theoremstyle{definition}
\newtheorem{definition}[theorem]{Definition}
\theoremstyle{remark}
\newtheorem{remark}[theorem]{Remark}
\newtheorem{example}[theorem]{Example}

\newcommand{\norm}[1]{\Vert#1\Vert}

\newcommand{\abs}[1]{\vert#1\vert}

\DeclareMathOperator{\grad}{\nabla}
\DeclareMathOperator{\argmin}{{argmin}}

\DeclareMathOperator{\dist}{{dist}}
\newcommand{\bq}{\begin{equation}}
\newcommand{\eq}{\end{equation}}

\newcommand{\R}{\mathbb{R}}

\newcommand{\Rn}{\R^n}
\newcommand{\e}{\epsilon}
\newcommand{\bO}{\mathcal{O}}

\newcommand{\Dir}{\mathcal{D}}

\newcommand{\axu}{a(x,u(\cdot))}
\newcommand{\uc}{u(\cdot)}

\begin{document}

\title[quasiconvex envelope]
{Computing the quasiconvex envelope using a nonlocal line solver}

\author{Bilal Abbasi \and Adam M. Oberman}
\address{Adam M. Oberman
\hfill\break\indent
Department of Mathematics and Statistics
\hfill\break\indent
McGill University 
\hfill\break\indent
{\tt adam.oberman@mcgill.ca}}

\date{\today}
\begin{abstract}
  Recently in a series of articles, Barron, Goebel, and Jensen \cite{barron2012functions} \cite{barron2012quasiconvex} \cite{barron2013quasiconvex} \cite{barron2013uniqueness} have studied second order degenerate elliptic PDE and first order nonlocal PDEs for the quasiconvex envelope.  Quasiconvex functions are functions whose level sets are convex.    The PDE is difficult to solve.  In this article we present an algorithm for computing the quasiconvex envelope (QCE) of a given function.  The QCE operator is a level set operator, so this algorithm gives a method to compute convex hull of sets represented by a level set functions. 
We present a nonlocal line solver for the quasiconvex envelope (QCE), based on solving the one dimensional problem on lines.  We find an explicit formula for the QCE of a function defined on a line. 
\end{abstract}

\maketitle

\tableofcontents

\section{Introduction}

In this article we present a numerical method to find the quasiconvex envelope of a given function.  Quasiconvex (QC) functions are functions whose sublevel sets are convex.    The related PDE is a level set PDE \cite{sethian1999level, Osher88frontspropagating} which means that the solution of this problem allows us to solve the following problem.
\begin{problem}
Given a level set representation of a set $S$, $S  = \{ g(x) \le \alpha \}$ find $u(x)$ so that $\{ u \le \alpha \}$ is the convex hull of $S$.
\end{problem}
 Going from the level set representation, taking a convex hull using standard algorithms, and then going back to the level set representation is costly.  The quasiconvex envelope solves this problem directly.  Applications of convex hulls of level sets functions include collision testing (see  \cite{mcadams2010crashing}). In other cases, for example the wearing of a convex stone \cite{firey1974,ishii2003}, the evolution of the front occurs only on the convex hull.  Quasiconvexity also appears in economic theory, economic theory \cite[Chapter 4]{avriel1988generalized} and specifically in ~\cite{arrow1961quasi} in the context of convex preferences between bundles of goods.  Convexity of levels sets of solutions of PDEs is also a topic of interest~\cite{caffarelli1982convexity, kawohl1985rearrangements, colesanti2003quasi}.  As far as we know, there has not been any previous work on computing QC envelopes.  An early paper related to convexity of functions and level sets is \cite{vese}. Convergence rates for parabolic equations to convex envelopes envelopes was studied in \cite{carlier2012}.

Recently in a series of articles, Barron, Goebel, and Jensen \cite{barron2012functions} \cite{barron2012quasiconvex} \cite{barron2013quasiconvex} \cite{barron2013uniqueness} have studied second order degenerate elliptic PDE and first order nonlocal PDEs for the quasiconvex envelope.   They present a second order degenerate elliptic obstacle problem for quasi-convex envelope.  However the PDE does not have unique solutions.  In order to remedy this problem,  they consider robustly  quasiconvex functions, which remain quasiconvex under small linear perturbations.   

Previously, wide stencil finite difference schemes have been used to solve related PDEs, including convex envelopes \cite{ObermanConvexEnvelope, ObermanCENumerics}, and directional convex envelopes \cite{ObermanR1CE}.  We first considered applying these methods to the PDE proposed by Barron, Goebel and Jensen.  However, it is challenging to discretize the small sets of relevant directions used in the operators. 

As an alternative, in \cite{barron2012quasiconvex} a nonlocal first order operator for the envelope in proposed. The method we propose here is also nonlocal and first order, but it is different.
 We present a nonlocal line solver for the quasiconvex envelope (QCE), based on solving the one dimensional problem on all lines.  We find an explicit formula for the QCE of a function defined on a line. It is the solution of a nonlocal first order nonlinear PDE for the QCE.   Numerically, the solution can be found efficiently by a fast sweeping or marching method \cite{sethian1999fast, tsai2003fast} .  Further we show that if we iteratively solve this PDE on every line, we obtain the QCE.  In practice, we solve on a grid, using a finite set of directions for lines at each grid point.  In contrast our previous works on wide stencil schemes, we are not restricted to grid directions: since the operator is first order, we can solve in arbitrary directions.   However, the cost of solution increases with the number of directions used.   Since we do not check every line, we cannot ensure that the grid function is QC.  However, we make an argument that it should be approximately QC, with improvements as the number of directions increases.   
 
 Barron, Goebel and Jensen also proposed the notion of robust-quasiconvex functions, which are functions which remain QC under small linear perturbations.  By making a small modification to our one-dimensional problem,   which does not affect the efficiency, we can compute the robust-QCE on lines, and extend the idea to robustly-QC envelopes in higher dimensions. 
  
The contents of this article are as follows.  In \S\ref{sec:Background} we give relevant definitions and background.  In \S\ref{sec:QC1d} we derive a nonlocal PDE for the QCE in one dimension, then we give an expression for the solution operator.  We also prove that the solution operator satisfies a comparison principle, and is a non-decreasing map.  In \S\ref{sec:RobustQC} we modify the one dimensional PDE to robust QCE, by adding a single term.   In \S\ref{sec:HigherDim} we show how to extend the one dimensional operator to higher dimensions, but applying it along lines.  We prove convergence of the iterative method to the QCE.  We also discuss directional resolution, which comes from using only a finite number of directions.  In \S\ref{sec:NumRes} we present numerical results for the QCE and the robust-QCE, in two and three dimensions.  We show that the solver is fast, we the loss of convexity with finite directions, and we show how to overcome this with the robust-QCE.  
  
 

\section{Quasiconvex Functions}\label{sec:Background}

\subsection{Background}
Write 
$
S_\alpha(u)\equiv\{x \in \Rn \mid u(x)\leq\alpha\}
$
for the $\alpha$-sublevel set of a function $u$.
\begin{definition}
The function $u: \Rn \to \R$ is quasiconvex if every sublevel set $S_\alpha(u)$ is convex. 
Given $g:\Rn \to \R$, continuous and bounded below,  the quasiconvex envelope of $g$, $QCE(g)$, is defined to be 
\begin{equation}
	\label{QCE}
QCE(g)(x) = \sup \{ v(x) \mid v(y) \le g(y) \text{ for every  $y\in \Rn$}, \quad \text{ $v$ quasiconvex.}  \} 
\end{equation}

\end{definition}

\begin{lemma}
The function $u: \Rn \to \R$ is quasiconvex if and only if  
\begin{equation}
\label{QCcond0}
u(tx + (1-t)y) \le \max \{ u(x), u(y) \}, 
\qquad \text{ for all } x,y \in \Rn,  0 \le t \le 1. 	
\end{equation}
\end{lemma}

\begin{proof}
The set $S$ is convex if whenever $x$ and $y$ are in $S$ then so is the line segment $tx + (1-t)y$, $t \in [0,1]$. 
Thus convexity of $S_\alpha(u)$ is equivalent to condition 
\[
u(x), u(y) \le \alpha \implies u(t x + (1-t)y) \le \alpha,
\quad \text{ for all } 0 \le t \le 1
\]
which is equivalent to \eqref{QCcond0}.
\end{proof}
Figure (\ref{fig:CEandQC1d}) provides a illustrates the convex envelope and the quasiconvex envelope.  It also illustrates the robust-QCE, defined below. 

The following example shows that,   at least when the function has flat parts, local conditions are not enough (even in one dimension) to determine if a function is quasiconvex. 
\begin{example}[Local conditions are not enough]\label{ex:nonlocal}
Let $C$ be a closed, bounded convex set, and let $u(x) = \dist(x,C)$ be the distance function to $C$.  Then $u$ is convex, and the sublevel sets are as well, so $u$ is quasiconvex.  On the other hand the sublevel sets of $v(x) = -u(x)$ are not convex, but unless $C$ consists of a single point, the only way to see this is by taking a triplet $tx+(1-t)y,x,y$ in \eqref{QCcond0} where $tx+(1-t)y$ is in $C$ and $x,y$ are outside $C$.  In particular, in one dimension, the function $v(x) = -\dist(x,[-1,1])$ is not QC, but we can only check this using points which are far apart.  
\end{example}

\begin{figure}[t]
    \centering
    \begin{subfigure}[b]{.5\textwidth}
    \centering
        \includegraphics[width=.95\textwidth]{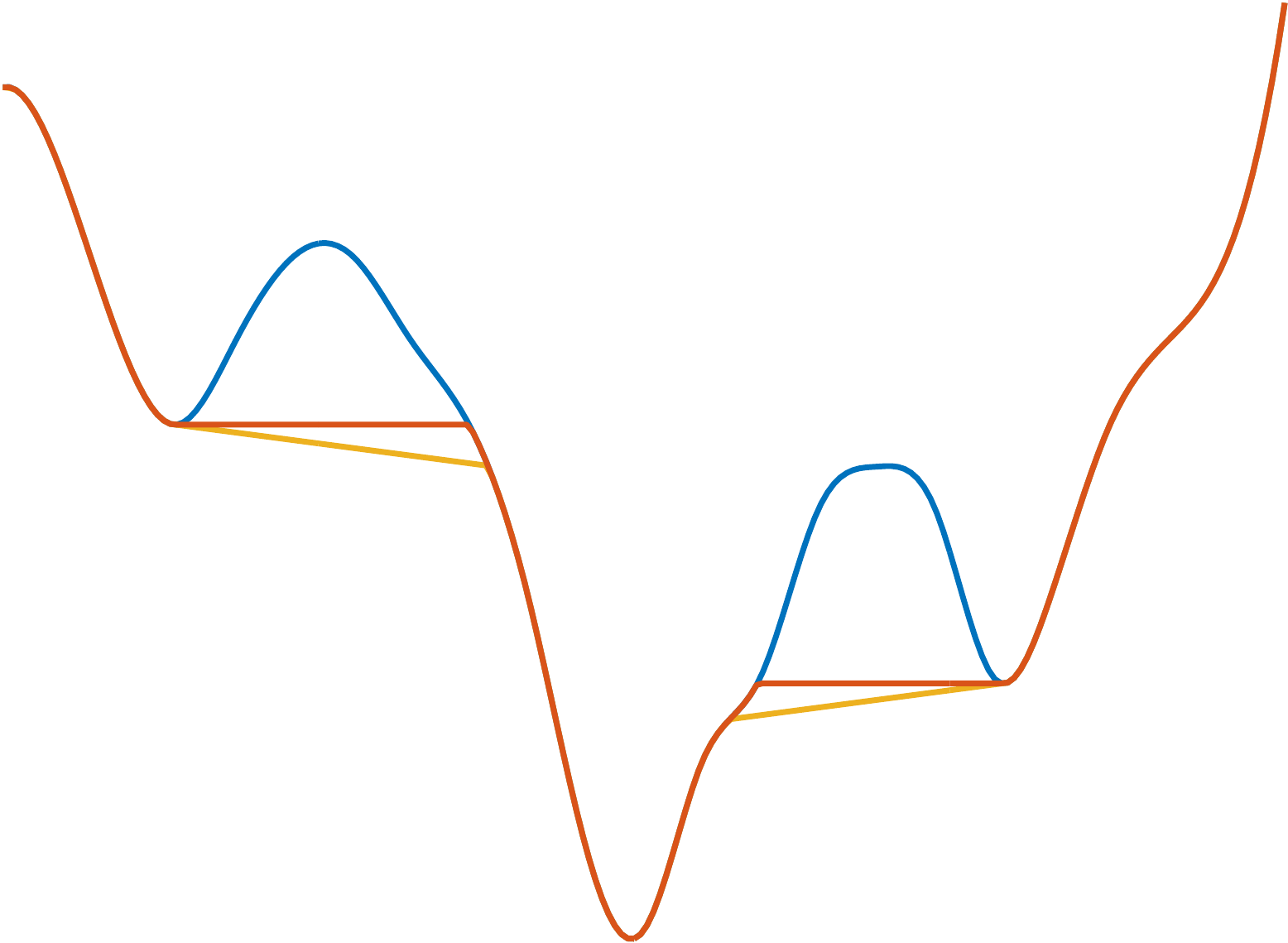}
        \caption*{(robust) quasiconvex envelope}
    \end{subfigure}
    ~
    \begin{subfigure}[b]{.5\textwidth}
    \centering
        \includegraphics[width=.95\textwidth]{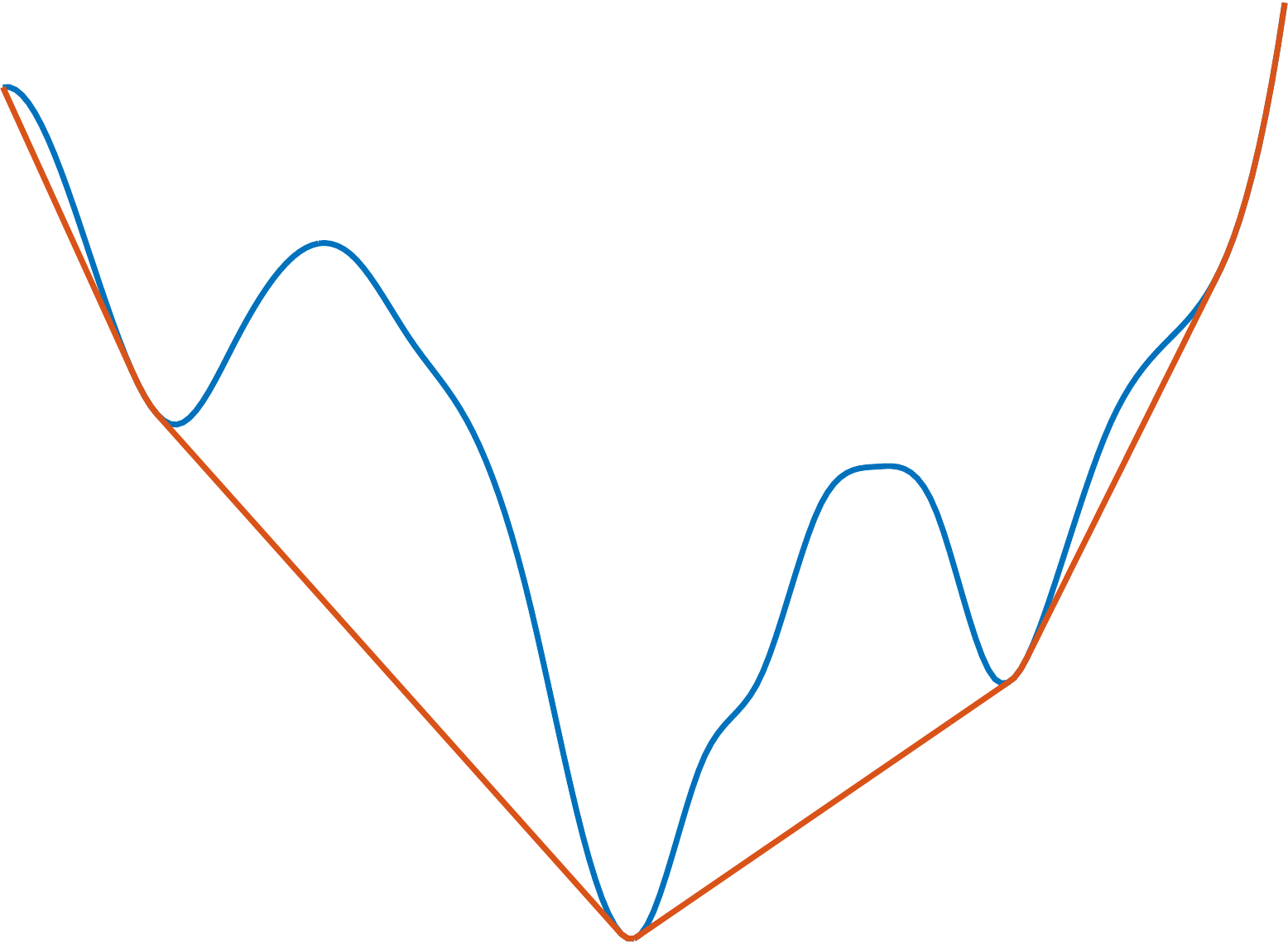}
        \caption*{convex envelope}
    \end{subfigure}
    \caption{A function (blue) and its envelopes.  Left: the quasiconvex (red) and robustly quasiconvex (yellow) envelope. Right: the convex envelope (red).}
    \label{fig:CEandQC1d}
\end{figure}

\begin{definition}
\label{Defn:dirQC} The function 
$u:\mathbb{R}^n\rightarrow\mathbb{R}$ is quasiconvex along the line $\ell = \{ x+ tv \mid t\in \R\}$ if the restriction of $u$ to  the line $\ell$ is quasiconvex.  
\end{definition}
\begin{proposition}
\label{dirQC}
The function $u$ is quasiconvex on $\Rn$ if and only if $u$ is quasiconvex on every line.  
\end{proposition}
\begin{proof}
This is clear from \eqref{QCcond0}. 
\end{proof}

\subsection{One dimensional characterization of QC}
In one dimension we recall the following simple characterization of quasiconvex functions. Refer to Figure~\ref{fig:QC1d}, which also illustrates the algorithm which follows.  
\begin{definition}[Increasing, Decreasing, and Down-Up functions]
Let $I = [a,b]$ be a bounded interval in $\R$. Write $C(I)$ for the set of continuous functions on the interval $I$.
Suppose $u \in C(I)$.  
We say $u$ is (nonstrictly) increasing, and write,
\[
u\in CI^+(I), \quad     \text{if $x < y$ implies $u(x) \leq u(y)$.}  
\]
We say $u$ is (nonstrictly) decreasing, and write 
\[
u\in CI^-(I), \quad \text{ if $x < y$ implies $u(x) \geq u(y)$.}
\]
 We say $u$ is down-up if there exists a global minimizer $x^*$ of $u$ and if the restriction of $u$ to $[a,x^*]$ is decreasing and the restriction of $u$ to $[x^*,b]$ is increasing. 
\end{definition}

\begin{proposition}\label{propDownUp}
Suppose $u:\R\rightarrow\R$ is continuous and bounded below. Let $x^*$ be a global minimizer of $u$. Then $u$ is quasiconvex if and only if $u$ is down-up.
\end{proposition}

\begin{proof}
 
Step 1: necessity. Assume $u$ is quasiconvex.   Choose $x \not = x^*$. 
Since $u$ is QC, the sublevel set $\{ y \mid u(y) \leq u(x) \}$ is convex and contains $x^*$.  In particular, it is an interval with $x$ at one endpoint.  Assume for now that $x$ is the left endpoint.  This means that $u(x)\leq u(x-h)$ for $h>0$ so taking $h \to 0$ we see that $u$ is (nonstrictly) decreasing at $x$.  Similarly if $x$ is the right endpoint, then $u(x)\leq u(x+h)$ for $h>0$. Similarly, $u$ is nonstrictly increasing for at $x$ in this case.

Step 2: sufficiency. Suppose, for contradiction,  (ii) holds for  $u$,  but  $u$  is not quasiconvex. 
Then there exists $x,y,z$ such that $x\in(y,z)$ satisfies $u(x)>\max\{u(y),u(z)\}$.   We can assume that $x^*$ is outside the interval $(y,z)$, since, if not, we can shrink the interval.
First suppose that $x^* > z$.  Then since $u(y) < u(x)$, there is a point in $(y,z)$ at which $u$ is increasing. This contradicts our assumption (ii). Next if $x^* < y$, we can make a similar argument using the interval $(x,z)$ and obtain a similar contradiction.
\end{proof}

\subsection{Nonlocal PDE for quasiconvexity}

For continuously differentiable quasiconvex functions, one can derive a characterization which is analogous to the supporting hyperplane condition of convex functions. This is obtained by taking the limit $t\to 0$ in $u(x+t(y-x))-u(x) \ge 0$, when $u(x) \le u(y)$, yielding the following necessary and sufficient condition
\bq\label{QCfirstorder}
u(x) \le u(y) \implies  \grad u(x) \cdot (y-x) \ge 0
\eq
This  condition can be extended to continuous functions, when interpreted in the viscosity sense~\cite{barron2012quasiconvex}.  This leads to a nonlocal Hamilton-Jacobi equation which characterizes QC functions:
  $H[u](x) \le 0$ for all $x \in \Omega$, in the viscosity sense 
  where 
\begin{equation}
\label{HJnd}
H[u](x)\equiv H(x,u,\grad u(x)) \equiv\sup_{ \{ y\in\Omega \mid  u(y)\leq u(x)\} } \grad u(x) \cdot (y-x) 	
\end{equation}

\begin{remark}The nonlocal operator is costly because evaluating it at each point $x$ involves computing a local gradient against a large number of points $y$.
\end{remark}

\section{QCE in one dimension}\label{sec:QC1d}

\subsection{A nonlocal PDE for QC in one dimension}
In this section we give a PDE obstacle version of the one dimensional QCE.

\begin{definition}\label{axumin}
Given a continuous function $u	\in C(I = [a,b])$, define
\begin{align*}
	u_m& = \min_{ x\in I } u(x)
\\	I_m &= \argmin_{x \in I} u(x)
\\	x_l(\uc) & = \min I_m 
\\	x_r(\uc) &= \max I_m	
\end{align*}
Define the function $\axu: I \to \{-1, 0, 1\}$ by 
\begin{equation}\label{axu}
\axu =
\begin{cases}
	-1, & x < x_l
	\\ 0,& x \in [x_l, x_r]	
	\\+1, & x > x_r
\end{cases}	
\end{equation}
\end{definition}

\begin{definition}
Suppose $u: \R \to \R$  is continuous and bounded below. Define $\axu$, $x_l, x_r, u_m$ according to Definition~\ref{axumin}. We say that 
\begin{equation}\label{eqQCcond}
\axu u'(x)\leq 0	
\end{equation}
holds in the viscosity sense if:
(i) for any smooth function $\phi$,  whenever $x$ is a local maximum of $u-\phi$, then 
\[
\axu \phi'(x)\leq 0
\]
and 
\begin{equation}\label{um}\tag{ii}
u(x) = u_m, \quad \text{for all } x\in [x_l, x_r]	
\end{equation}
\end{definition}

\begin{proposition}
\label{firstorderQC}
Suppose $u:\R\rightarrow\R$ is continuous and bounded below.  Then $u$ is quasiconvex
if and only if \eqref{eqQCcond} holds in the viscosity sense,
\end{proposition}

\begin{proof}
Since $u$ is continuous and bounded below, it has at least one global minimizer, $x^*$. 

Step 1. Assume $u$ is quasiconvex, we wish to show \eqref{eqQCcond} holds.  Suppose $u-\phi$ has a local maximum at $x$.  We can assume that $\phi(x)=u(x)$.  Then $\phi(y)\geq u(y)$ in a neighbourhood of $x$. 

Since $u$ is QC, the sublevel set $\{ y \mid u(y) \leq u(x) \}$ is convex and contains $[x_l, x_r]$.  In particular, it is an interval which contains  $x$.   Assume for now that $x$ is the left endpoint. 
This means that $u(x)\leq u(x-h)$ for $h>0$ small. Then
\[
\phi(x)=u(x)\leq u(x-h)\leq \phi(x-h) \implies \phi(x)-\phi(x-h)\leq 0.
\]
Dividing by $h$ and taking $h\rightarrow0$ yields $ \axu \phi^\prime(x)\leq0$. 

Similarly if $x$ is the right endpoint, then $u(x)\leq u(x+h)$ for $h>0$ small. Hence by a similar calculation we have $\phi^\prime(x)\geq0$.

Finally, if $ x\in (x_l, x_r)$ then the sublevel set is simply $[x_l, x_r]$, so $u$ is constant near $x$ and if $u-\phi$ has a local maximum at $x$, then $\phi'(x) = 0$. 

In each case,  $\axu \phi^\prime(x) \leq0$ holds, as desired.  

Step 2. Suppose, for contradiction,  that \eqref{eqQCcond} holds for  $u$,  but  $u$  is not quasiconvex. 
Then there exists $x,y,z$ such that $x\in(y,z)$ satisfies $u(x)>\max\{u(y),u(z)\}$.

First suppose that $z < x_l$.  Define  $\phi(x)$ to be linear with slope $0 < m = (u(x)-u(y))/(x-y)$. There is a point $y' \in (y,x)$ which is a local max of $\phi- u$.  
But then  $\axu \phi_x(y') = m  >  0$ which means that \eqref{eqQCcond} does not hold.    

Next if $x_r < y$, we can make a similar argument using the interval $(x,z)$ and obtain a similar contradiction.

Finally, if $(y,z)$ overlaps with $(x_l, x_r)$, we must have $x$ on one side of the interval, since $u = u_m < u(x)$ on $(x_l, x_r)$ so we can redefine one of $y,z$ to recover the previous case. 

So $u$ must be quasiconvex. 
\end{proof}

\begin{proposition}\label{prop:QCE1d}
Suppose $g: [a,b] \subset \R \to \R$ is continuous. 
Define $g_m, I_g$ according to Definition~\ref{axumin}. Define $\axu$ by \eqref{axu}.  
The quasiconvex envelope of $g$ is the viscosity solution of
\begin{equation}\label{1dobs}
\tag{1DOb}
\max\{u(x)-g(x), \axu u_x(x)\} = 0,
\quad x \in (a,b)\setminus I_g
\end{equation}
along with 
\begin{equation}\label{Dirichlet1da}
u(a) = g(a), 
\qquad
u(b) = g(b),
\qquad
u(x) = g_m, 	\quad x \in I_g
\end{equation}
The increasing and decreasing envelopes, $u^+, u^-$ of $g$ are viscosity solutions of 
\begin{align}
	\max\{u^+(x)-g(x),  ~~u^+_x(x)\} = 0
	\\
	\max\{u^-(x)-g(x), -u^-_x(x)\} = 0
\end{align}
respectively, along with the boundary conditions \eqref{Dirichlet1da}.

\end{proposition}

\begin{proof}
By the Perron formulation, the supremum of all subsolutions of \eqref{1dobs} is the unique viscosity solution. However, by Proposition~\ref{firstorderQC}, the set of all subsolutions of \eqref{1dobs} are also those functions which are quasiconvex and bounded above by $g$, along with boundary conditions \eqref{um}. Hence the two feasible sets are identical and so the two formulations coincide.
\end{proof}

\subsection{Solution operator for $QCE(g)$  in one dimension}
We now give an explicit solution formula for the QCE in one dimension. 


\begin{definition}[Increasing and decreasing envelopes]
For $g \in C(I)$, 	We define the increasing and decreasing envelopes of $g$ to be 
\begin{align*}
	CIE^+(g)(x) &= \sup \{ v(x) \mid v \leq 	g, \quad  v \in CI^+(I)  \} 
	\\
	CIE^-(g)(x) &= \sup \{ v(x) \mid v \leq 	g, \quad  v \in CI^-(I)  \} 	
\end{align*}
\end{definition}

\begin{definition}
Suppose $g: [a,b] \subset \R \to \R$ is continuous. 
Define the solution maps $S^+, S^-, S^0 : C[a,b] \to C[a,b]$ by 
\begin{equation}\label{Sdefn}
\begin{aligned}
S^+(g)(x) &= \min\{ g(y)  \mid { y \geq x } \}
\\
S^-(g)(x) &= \min\{ g(y)  \mid { y \leq x } \}
 \\
S^0(g)(x) &= \max \{ S^+(g)(x), S^-(g)(x) \}.
\end{aligned}
\end{equation}

\end{definition}

\begin{proposition}[The increasing, decreasing, and quasiconvex envelope operators]
Let $g\in C(I)$.  Then $S^+(g) \in C^+(I)$,  $S^-(g) \in C^-(I)$ so the function increasing, and  decreasing, respectively.  
Furthermore, $S^0(g)$ is up-down, and each of the operators is an envelope operator, in particular,
\[
u = QCE(g)(x) = S^0(g)
\] 	
\end{proposition}

\begin{proof}
The first statements are clear from the definition. 

It is clear from the definition of $u$, that $u$ is non-increasing for $x \leq x^*$ and  nondecreasing for $x \geq x^*$ so by Prop~\ref{propDownUp}  above, $u$ is quasiconvex. It is also clear from the definition that $u \leq g$ and that $u(a) = g(a)$ and $u(b) = g(b)$.  

Next we claim that if $u(y) < g(y)$ then there is an interval $c < y < d$ where $u$ is constant on $[c,d]$ with $u = g$ at $c,d$, as in the bumps on Figure~\ref{fig:CEandQC1d}.  To prove the claim, first suppose that $x^* < y$.  Then $u(y) = u^+(y)$ and $u^+(y) = g(d)$ for some $d>y$.   Then by continuity of $g$, since $x^* < y$, there is some point $c < y$ where $g(c) = g(d)$.  Next, suppose that $x^* > y$.  Then a similar argument applied to $u^-$ gives the claim.  In addition, it is not possible that $y = x^*$, since $u(x^*) = g(x^*)$. 

Now we show that $u = QCE(g)$. 
Suppose for contradiction, that it is not.  This means that there exists a quasiconvex function $v(x)$,  with $v \leq g$, so that $v(y) > u(y)$ for some $y \in [a,b]$.   Clearly, $u(y) < g(y)$.  
Applying the previous claim, there are points $c < y< d$ where $g(c) = g(d) = u(y)$.  Since we assumed $v \leq g$ we have $v(c) \leq g(c) = u(c)$ and $v(d) \leq g(d) = u(d)$.  But $v(y) > u(y)$ contradicts the assumption that $v$ is QC.   So $u$ is indeed the QC envelope. 
\end{proof}

\begin{definition}[Comparison Principle]
Let $S:C(I) \to C(I)$ be a map from continuous functions to continuous functions.  For $u,v \in C(I)$, write 
\[
u \leq v \text{ if }  u(x) \leq v(x), \text{ for all $x\in I$}.  
\]
The comparison principle holds for $S$ if 
\[
u \leq v \implies  S(u) \leq S(v).
\]	
The map $S$ is non-increasing if 
\[
S(u) \leq u
\]
\end{definition}

\begin{proposition}\label{Comparison1d}
The comparison principle holds for $S^+, S^-$ and $S^0$, and they are non-increasing maps. In particular, 
\[
	u \leq v 
	\implies
	QCE(u) \leq QCE(v) 
	\]
\end{proposition}

\begin{proof}
	This follows from the explicit formulas \eqref{Sdefn}.   In particular, if $u\leq v$ then $\min_{ y \geq x } u(y) \leq \min_{ y \geq x } v(y)$ since the comparison holds at each point.  So comparison holds for $S^+$.  Similarly comparison holds for $S^-$.   Next, since comparison holds for $S^+$ and $S^-$,  $u \leq v$ implies that  $S^+(u)(x) \leq S^+(v)(x)$ and also $S^-(u)(x) \leq S^-(v)(x)$  so 
	\[
	\max \{ S^+(u)(x), S^-(u)(x) \} \leq  \max \{ S^+(v)(x), S^-(v)(x) \}
	\] which shows comparison for $S^0$.
\end{proof}

\subsection{A fast marching/sweeping method}
We can implement the one dimensional  QC-envelope on a line segment discretized by the points $x_0 < x_1 < \dots < x_N$ using a fast sweeping method \cite{zhao2005fast} or a fast marching method \cite{sethian1999fast}.

Set $u^-(x_0) = g(x_0)$ and define inductively (sweeping from left to right)
\[
u^-(x_{j+1}) = \min\left\{ u^-(x_j), g(x_{j+1})\right \},
\quad j = 0,\dots, N-1.
\]
Similarly, set $u^+(x_N) = g(x_N)$ and define inductively (sweeping from right to left)
\[
u^+(x_{j-1}) = \min\left\{ u^+(x_j), g(x_{j-1})\right \},
\quad j = N,\dots, 1
\]
Then  define $u(x_j) = \max \{ u^+(x_j), u^-(x_j) \}$.  Refer to Figure (\ref{fig:QC1d}) for a visualization.
\begin{remark}
We can also make this a fast marching method, if we first find the minimum point, and then march in towards it.  This reduces the cost in half, once the minimum point is found. 	
\end{remark}

\begin{figure}[t]
\centering
	\includegraphics[width=.45\textwidth]{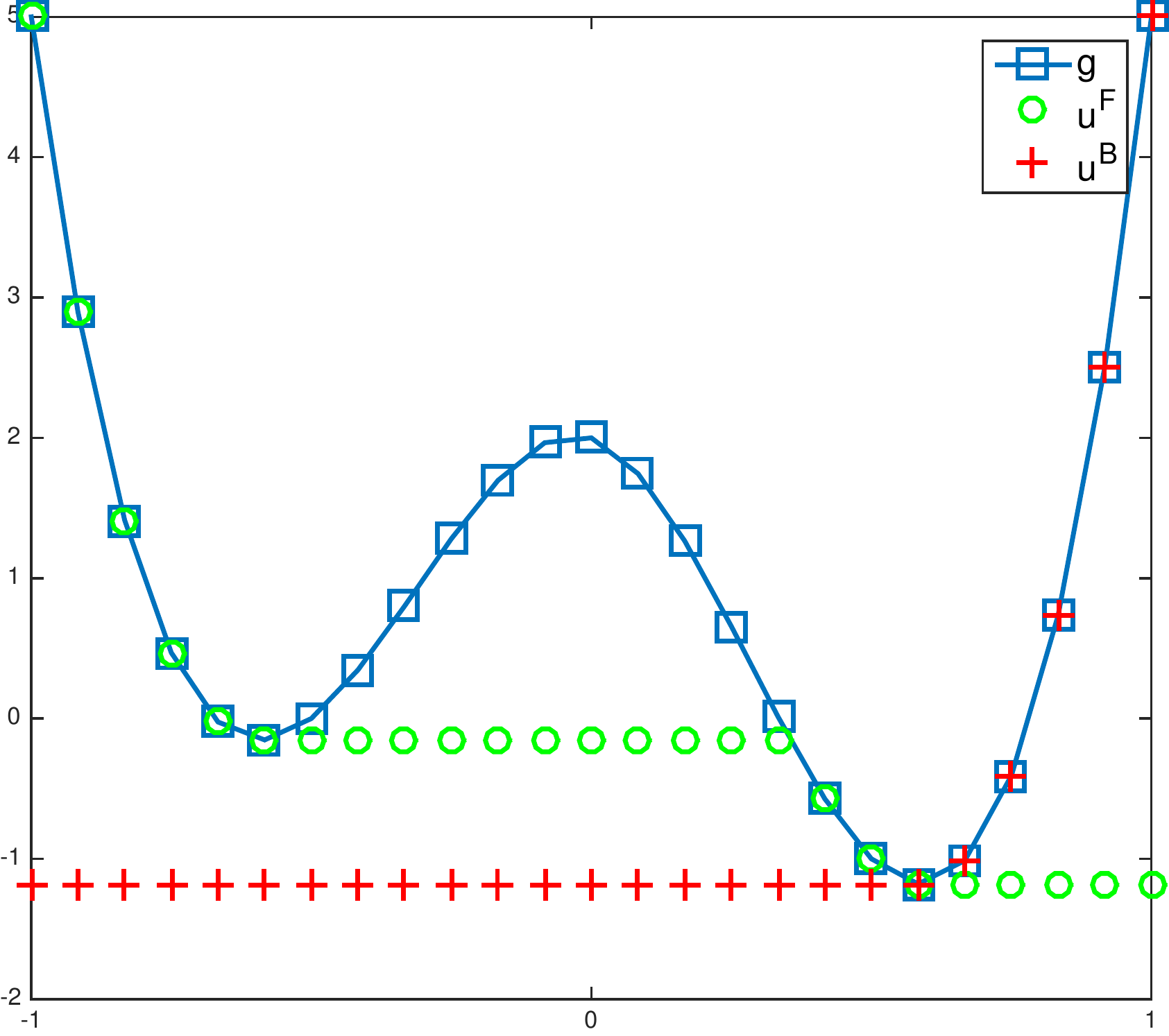}\hspace{.5cm}
	\includegraphics[width=.45\textwidth]{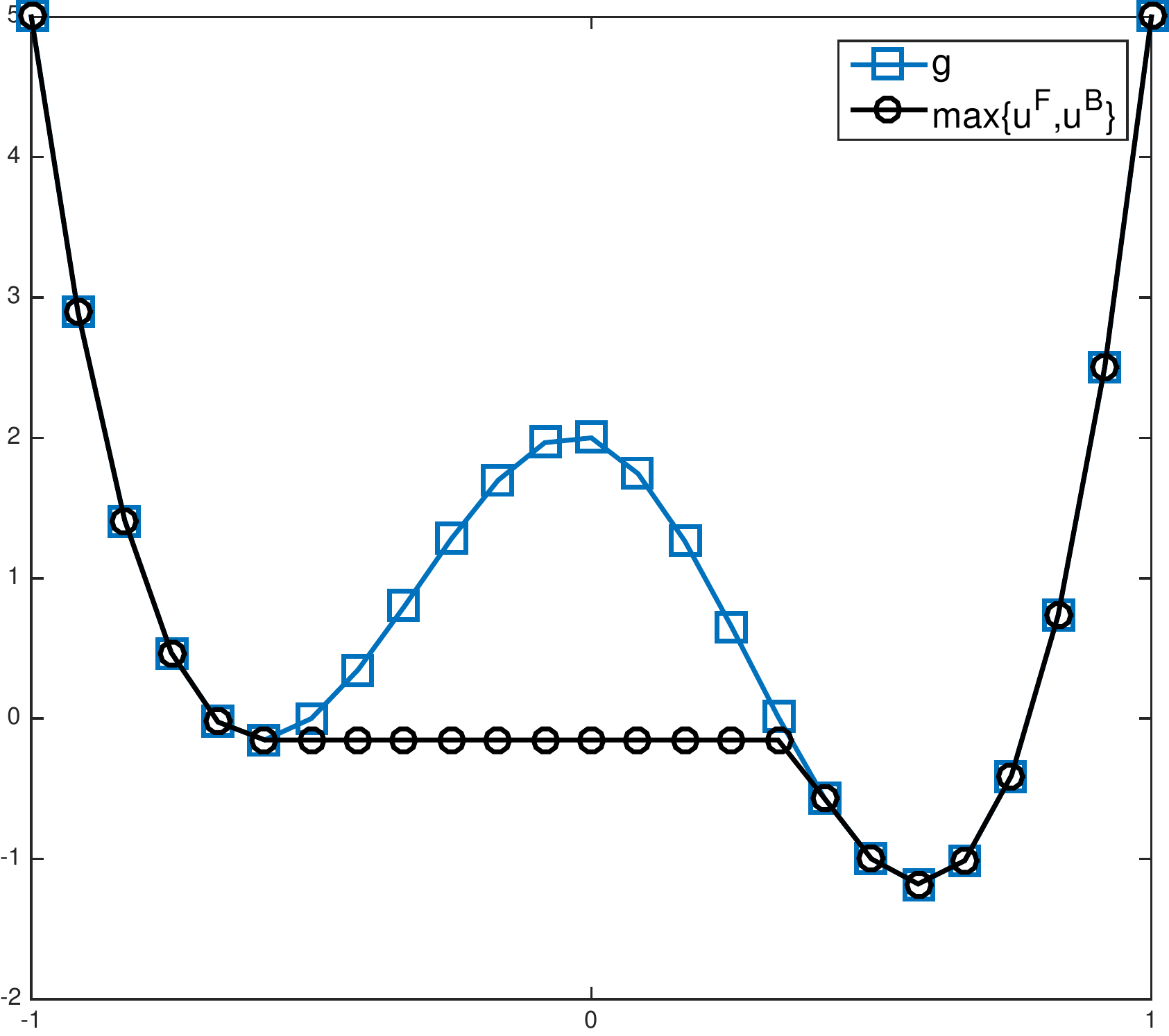}
\caption{Left: Forward and backward sweep give the increasing/decreasing envelopes of $g$. The quasiconvex envelope is the maximum of the two. }
\label{fig:QC1d}
\end{figure}

\section{Robust quasiconvexity}\label{sec:RobustQC}

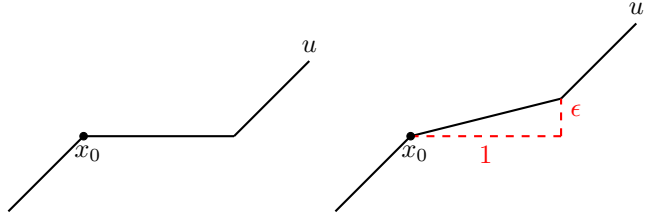
\begin{figure}[t]
\begin{tikzpicture}
\draw [black,thick] (0,0) -- (1,1);
\draw [black,thick] (1,1) -- (3,1);
\draw [black,thick] (3,1) -- (4,2);
\node [above,black] at (4,2) {$u$};
\draw[fill] (1,1) circle [radius=0.05];
\node [below,black] at (1,1) {$\ x_0$};
\end{tikzpicture}
~
\begin{tikzpicture}
\draw [black,thick] (0,0) -- (1,1);
\draw [black,thick] (1,1) -- (3,1.5);
\draw [black,thick] (3,1.5) -- (4,2.5);
\draw [red,thick,dashed] (1,1) -- (3,1);
\draw [red,thick,dashed] (3,1) -- (3,1.5);
\node [below, red] at (2,1) {$1$};
\node [above right, red] at (3,1.15) {$\epsilon$};
\node [above,black] at (4,2.5) {$u$};
\draw[fill] (1,1) circle [radius=0.05];
\node [below,black] at (1,1) {$\ x_0$};
\end{tikzpicture}
\caption{Robust quasiconvexity. The function on the left is not $\e$-robustly quasiconvex. The function on the right is. }
\label{viscosityinterpretation}
\end{figure}

Robustly-QC functions are stable under perturbations by linear functions with bounded slopes, which is not the case for QC functions.  For example, the function $u(x)$ on the left in \autoref{viscosityinterpretation} is QC, but $u(x) - \e x$ is not, for arbitrarily small $\e$.  Robustly-QC functions have the advantage over QC functions, that they can be characterized in the viscosity sense by a PDE operator~\cite{barron2012functions}.

\begin{definition}[$\e$-robustly QC]
	The function $u \in C(\Rn)$ is $\e$-robustly quasiconvex if $v_p(x) = u(x) + p\cdot x$ is QC for every $\abs{p} \leq \e$. 		
\end{definition}

\begin{lemma}
	The function $u \in C(\Rn)$ is $\e$-robustly quasiconvex if an only if it is $\e$-robustly QC on every line.
\end{lemma}

\begin{proof}
	Suppose $u \in C(\Rn)$ is $\e$-robustly quasiconvex.  Then for each $\abs{p} \leq \e$, $v_p$ is QC, and by Prop~\ref{dirQC}, $v_p$ is QC on every line.   
	
	Next suppose $u$ is $\e$-robustly QC on every line.   For each $\abs{p} \leq \e$, consider $v_p$.  When we restrict $v_p$ to the line $x = d t + x_0$, we obtain the function $u(dt + x_0) + p\cdot(dt + x_0)$.   This function is a perturbation of $u$ restricted to the line by $p\cdot d t + p\cdot x_0$ which has slope $\abs{p\cdot d} \leq \e$, since $\abs{p}\leq \e$ and $\abs{d} \leq 1$.   So the restriction is QC.   Since the restriction of $v_p$ to lines is QC for every line, and this holds for all admissible $p$, $u$ is $\e$-robustly QC. 
\end{proof}

The ideas of \S\ref{sec:QC1d} generalize naturally to this case.  In particular, we can define increasing and decreasing functions which grow by at least $\e$.   We can also define the resulting envelopes.  

The nonlocal PDE for the $\e$-robust envelope becomes
\begin{equation}
\max\{u(x)-g(x), \axu u_x(x) + \e \} = 0, 
\quad x \in (a,b)\setminus I_g
\end{equation}
along with \eqref{Dirichlet1da}.

The solution method becomes, on a grid of spacing $h$, 
\begin{align}
u^-(x_{j+1}) &= \min\left\{ u^-(x_j)-\e h, g(x_{j+1})\right \},
\quad j = 0,\dots, N-1.
\\
u^+(x_{j-1}) &= \min\left\{ u^+(x_j) + \e h, g(x_{j-1})\right \},
\quad j = N,\dots, 1
\end{align}
\begin{remark}
For the QCE, it was less important to enforce the Dirichlet boundary conditions on $I_g$, since enforcing it at one point of $I_g$ is enough.  However for the robust QCE, we need to make sure that the solution does not decrease below $g_m$. 
\end{remark}

\section{QCE in higher dimensions}\label{sec:HigherDim}
\subsection{Directionally Quasiconvex Functions}
Now we consider quasiconvex functions defined on $\Rn$ with $n> 1$.   We introduce a notion of quasiconvexity with respect to a direction set.  We recover the usual definition of quasiconvexity when the direction set becomes all directions. 

\begin{definition}
Let  $\Dir$ be a set of unit vectors in $\R^n$, we call is a direction set.   
The continuous function $u: \Rn \to \R$ is $\Dir$-QC (directionally QC) if
\begin{equation} \label{d.qc}
u(\lambda x + (1-\lambda) y) \le \max\{u(x),u(y)\}, 
\quad \text{ for all $0\le \lambda \le 1$, and all ${x-y} \in  \Dir$}.
\end{equation}
The $\Dir$-quasiconvex envelope of a given function $g$ is defined as the pointwise supremum of all $\Dir$-convex functions which are majorized by~$g$,
\bq\label{DCEdefn}
QC^{\Dir}(g)(x) = \sup\{  v(x) \mid  v(y) \leq g(y) \text{ for all } y, \quad v \text{ is $\Dir$-QC} \}.
\eq	
\end{definition}
\begin{remark}
	As is the case for QC function, $\Dir$-QC functions are closed the maximum operation.  Suppose $u,v$ are $\Dir$-quasiconvex functions.  Then so is $w(x) = \max(u(x),v(x))$.  Thus $u = QC(g)$ exists and it is unique, and it is $\Dir$-quasiconvex.
\end{remark}

\subsection{Approximate quasiconvexity}
In this section we want to study how far away from quasiconvex a directionally convex function can be.

Let $\Dir$  be a set of direction vectors. Define the \emph{directional resolution} 
\bq\label{eq:dir_res}
d\theta \equiv \max_{\abs{w}=1}\min_i\cos^{-1}(w^Td), \quad d \in \Dir 
\eq
to be the largest angle an arbitrary vector can make with any vector in $\Dir$. In two dimensions it is easy to see that $d\theta$ is simply half the maximum angle between any two vectors in $\Dir$.

Then a $\Dir$-QC function can have level sets with negative curvature, of size $\bO(d\theta)$. 

Formally speaking, consider the locally quadratic function, choosing coordinates so that $\grad u = (p,0)$, with $p > 0$, 
the we can write
\[
u(x,y) = px + ax^2/2 + bxy + cy^2/2
\] 
the curvature of the zero sublevel set is $c$ (meaning convex when $c \geq 0$).   
An elementary calculation shows that  
\[
c \geq - \abs{p + b} d\theta - ad\theta^2/2
\]

This calculation can be made more precise by consider QC-PDE operator from \cite{barron2013uniqueness}, and using viscosity solution techniques, as in \cite{carlier2012}.

\subsection{Enforcing quasiconvexity in a single direction}
The PDEs and solvers from the previous section can be extended to higher dimensions so that we can find increasing, decreasing, envelopes for an direction $d$.  This also allows us to find the envelope function which is QC \emph{in a single given direction $d$}.  By iterating this process we over different directions, we hope to find $QCE(g)$.

\begin{proposition}\label{prop:QCEd2}
Let $\Omega$ be a bounded convex domain in $\Rn$ with boundary $\partial \Omega$. 
Suppose $g: \Omega \subset \Rn \to \R$ is continuous. Let $d$ be a given unit vector in $\Rn$. 
The increasing and decreasing envelopes, $u^+, u^-$ of $g$ are viscosity solutions of 
\begin{align}
	\max\{u^+(x)-g(x),  d\cdot \grad u^+(x)\} = 0
	\\
	\max\{u^-(x)-g(x), -d\cdot \grad u^-_x(x)\} = 0
\end{align}
respectively, along with the boundary conditions
\begin{equation}\label{Dirichlet2d}
	u = g \text{ on } \partial \Omega
\end{equation}
The $d$-quasiconvex envelope of $g$, $u = QCE^d(g)$ is given by
\[
u(x) = \max\{ u^+(x), u^-(x)\}
\]
it is a viscosity solution of the nonlocal equation 
\begin{equation}\label{1dobs}
\max\{u(x)-g(x), (x^*(d,x)-x) \cdot \grad u(x)\} = 0
\end{equation}
where 
\[
x^*(d,x) \in \argmin \{ u(y) \mid y \in \Omega, y = x+td \}
\]
along with \eqref{Dirichlet2d}.
\end{proposition}
\begin{remark}
	Here the Dirichlet boundary conditions are to be understood in the viscosity sense.  In particular, the functions $u^+, u^-$  will not be continuous up to the boundary, as in the one dimensional case.  However they can be continuous in the direction of the characteristics.  
	For $QCE^d$, we take the minimum value along the line through $x$ with direction $d$, inside the domain. 
\end{remark}

\begin{remark}
	These PDEs can be solved explicitly using the method of characteristics.  The solutions are just the higher dimensional generalizations of the one dimensional operators given in \eqref{Sdefn}.  The same properties: comparison, decreasing, hold for the higher dimensional operators. 
\end{remark}

\subsection{Convergence of an iterative line solver in higher dimensions}
Now, in order to approximate the QC envelope, we use Proposition~\ref{prop:QCEd2}.  
We first choose a direction set $\Dir$, then for each direction $d$ we apply the result, solving using the method of characteristics on the grid. 
This is done by fast marching, or fast sweeping.   Since the characteristics are parallel straight lines,  the solution is computed in one sweep.    This can be done for any vector $d$, we are not restricted to grid direction vectors.  Then we iterate over a list of vectors $d$.  Define $S^\Dir$.  To be the solution operator for the full list of directions.   Then $S^\Dir$ satisfies the comparison principle, and it is decreasing. 

Denoting by $u^n$ the iterates, with $u^0 = g$
\[
u^n = S^\Dir(u^{n-1})
\] 
Write $u = QCE^\Dir(g)$.  Then $u  \leq u^0 = g \leq g$. By comparison, this means that 
\[
u \leq u^n\leq g
\]
Furthermore, the iterations are decreasing, in the sense that $u^{n+1} \leq u^n$.   
Moreover, we can establish convergence as in the proof of \cite[Theorem 3.1]{barron2012quasiconvex}: define the function $W(x) = \lim_{n\to \infty}u^n(x)$.  Then $u \leq W \leq g$.  Also, $W$ is $\Dir$-quasiconvex: if not, we can find points where the condition fails, and decrease $W$ further.   Then, since $u$ is the largest $\Dir$-QC function below $g$, we have $W = u$.

\subsection{Robust QC in higher dimensions}
The generalization of the robust QC to higher dimensions in analogously, using the line solver. 
\begin{remark}
	We can generalize further, replacing the requirement that $\abs{u_x} \geq \e$ by $\abs{u_x} \geq f(x)$ or even $\abs{u_x} \geq f(x,d)$ where $d$ is the direction for the directional line solver.   The effect of this term could be to introduce additional convexity in spatially or even directionally dependent manner.
\end{remark}

\section{Numerical results}\label{sec:NumRes}
In this section we present numerical experiments which validate the arguments presented earlier. Recall that the numerical solver, in one dimension, recovers the envelope after one sweep in each direction. \autoref{fig:CEandQC1d} was generated using the one dimensional solver for the QCE and the robust-QCE. 

We present several examples of the extension of the line solver to two dimensions, and conclude with an example in three dimensions. In the contour plots, the solid line represents the level sets of the original function (obstacle) and the dashed line represents the same level sets of the numerical solution. The two-dimensional numerical solutions shown are visualized on a $64\times64$ grid (larger grid sizes make the plots harder to read. We apply the line solver iteratively until a steady state is reached (taking a tolerance of $10^{-6}$).  Moreover, the direction sets used for computation were given by rational slopes, 
\[
\Dir^W =\{v\in\mathbb{Z}^2\mid \norm{v}_\infty\le W\}\setminus\{(0,0)\}
\]
in which case we say the direction set has width $W$.  We also performed computations using equally spaced direction vectors.

We begin with examples where the non-convexities are aligned with the grid. In this case, only one sweep along each direction is required to find the solution.  In general, after 10 or less iterations of the line solver in each direction, the solution was found, with a stopping criterion of $10^{-6}$. 

\begin{example}[Grid-aligned cones]
\label{ex:gridalignedcones}
We also consider examples of the type:
\[
g^{\theta,\alpha}(x)=\max\left\{\sqrt{(x_1+a^\theta_1)^2 + (x_2+a^\theta_2)^2},\sqrt{(x_1+b^\theta_1)^2 + (x_2+b^\theta_2)^2}-\alpha\right\}
\]
where $R(\theta)$ is the counter-clockwise rotation matrix by angle $\theta$, $a^\theta = R(\theta)(0.5,0)^\intercal $, and $b^\theta = R(\theta)(-0.5,0)^\intercal $. The vertical translation of the second argument by $\alpha$ adds an additional non-convexity by imposing asymmetry about the plane $x=0$. For now, we consider the case $\theta=0$, so that the non-convexities lie along horizontal lines. Results are found in \autoref{fig:gridaligned}.


\end{example}

\begin{example}[Non-convex signed distance function]
\label{ex:pacman}
Consider a signed distance function $g$, to the PacMan shape (a circle with one quadrant removed).
See \autoref{fig:gridaligned}.
\end{example}

\begin{figure}
\includegraphics[width=.4\textwidth]{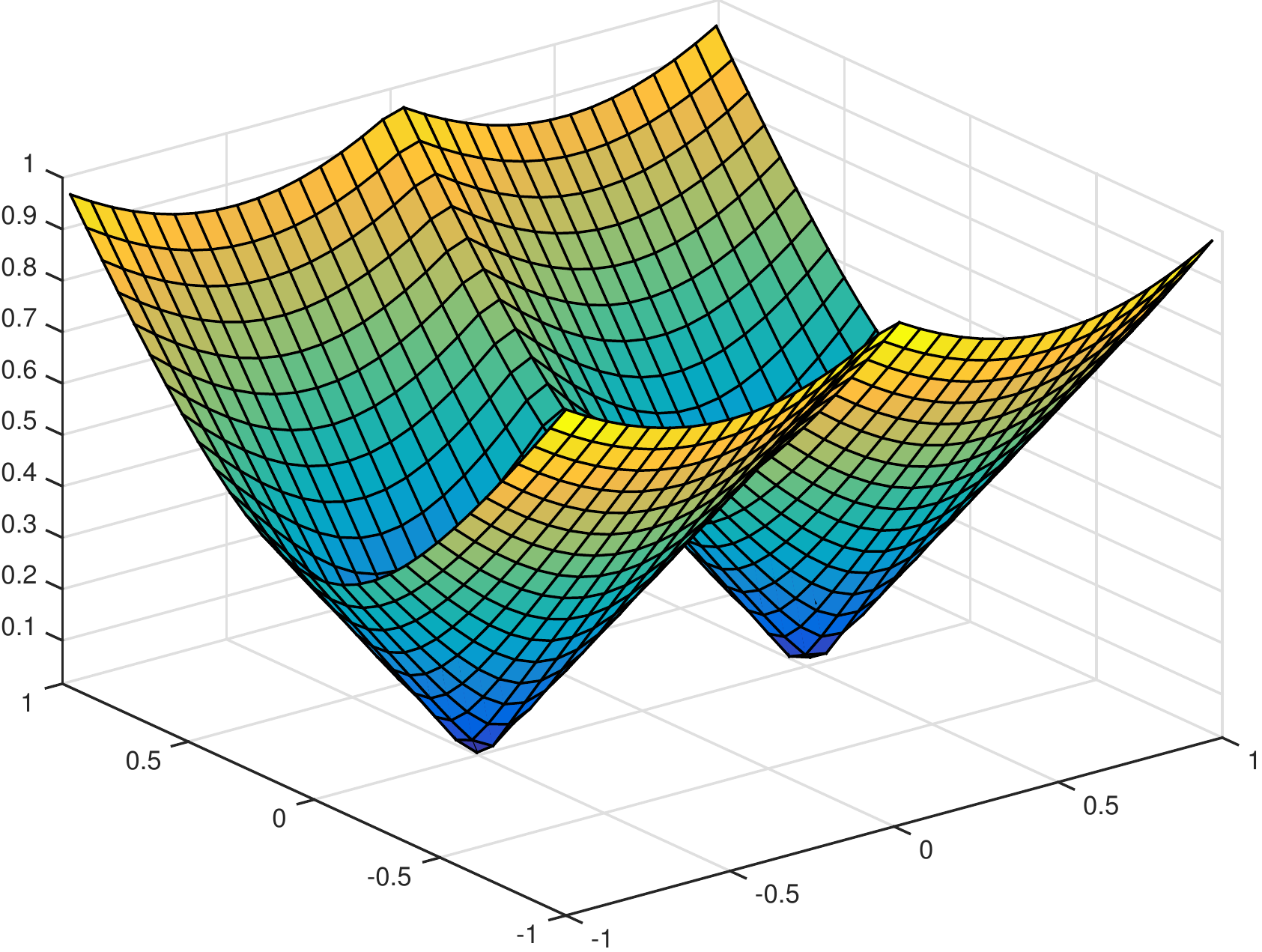}
\includegraphics[width=.4\textwidth]{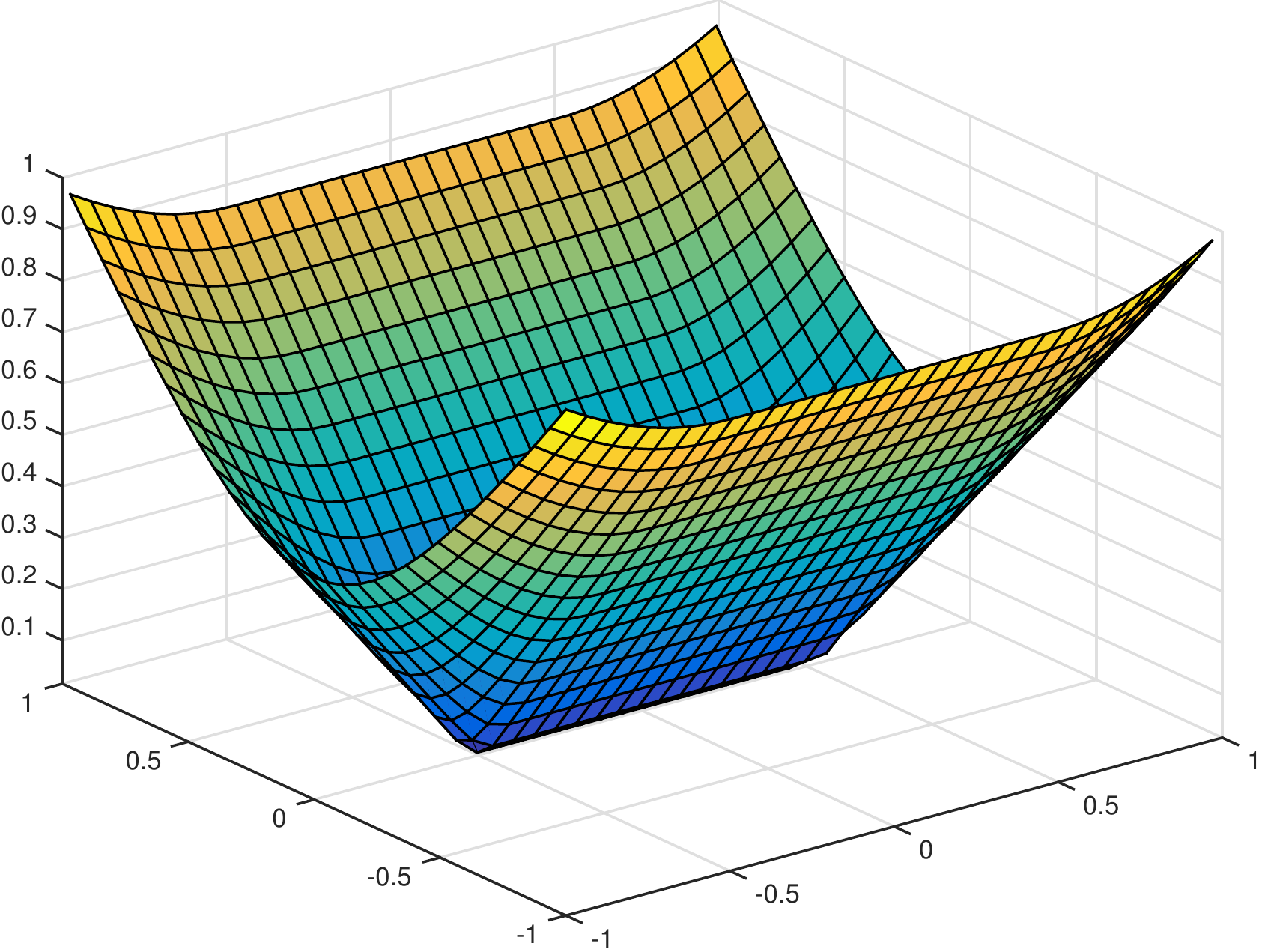}
\\
\vspace{1cm}
\includegraphics[width=.4\textwidth]{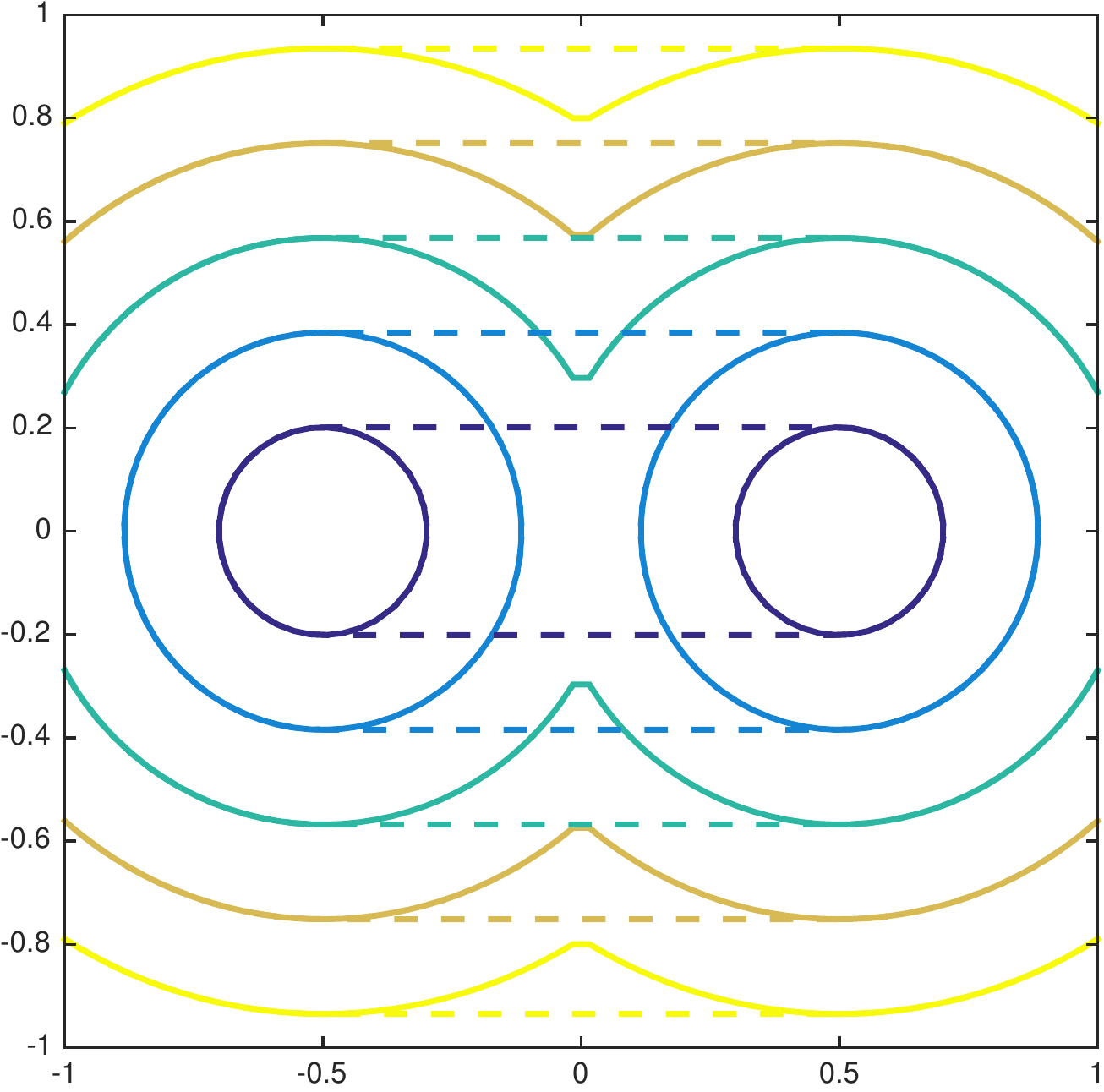}
\hspace{1cm}
\includegraphics[width=.4\textwidth]{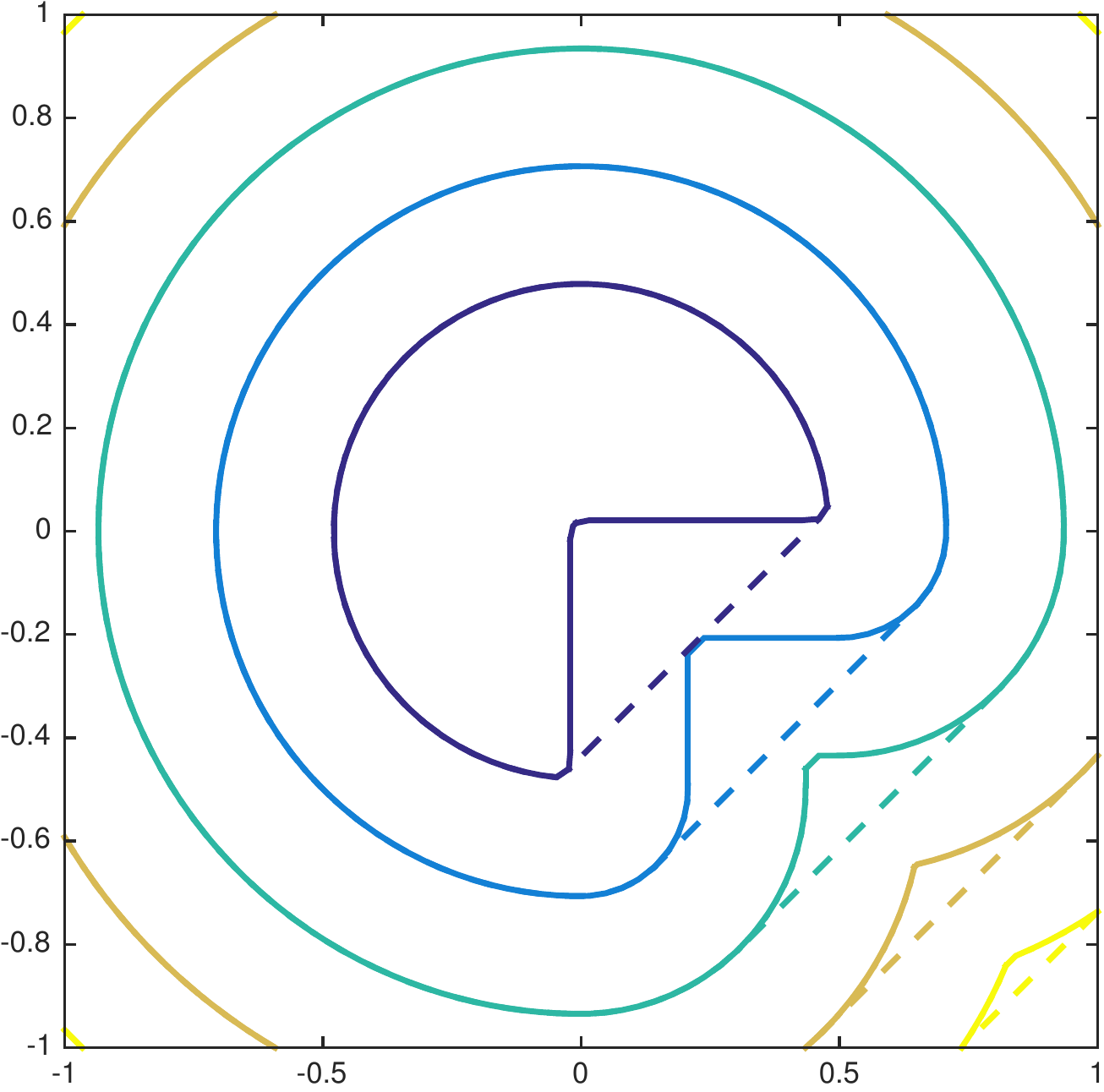}

\caption{Top: Example \ref{ex:gridalignedcones}: surface plot of $g$ at its quasiconvex envelope, $u$.
 Bottom Left: contour plot of $g$ (solid lines) superimposed with contour plot of $u$ (dashed lines).
 Bottom Right: contour plot for Example~\ref{ex:pacman}. 
}
\label{fig:gridaligned}
\end{figure}

%

\begin{example}[Non-grid aligned convexities]
\label{example_nonaligned}
We now consider the case where the non-convexities are not lined up with the grid, using the function $g^{\theta,\alpha}$ defined in Example \ref{ex:gridalignedcones}. 
We took  $\theta = \arctan(1/3)/2$ and $\alpha=0$. The rotations is the worst angle possible for the direction set of width 3.  As can be seen in \autoref{fig:nongridaligned}, the solution is visibly nonconvex. 
In practice, we can compute with more directions, but this example is for illustration. 
Next we used the $\e$-robust QCE to correct for the error in the previous example, with $\e=0.15$.  The level sets become convex, without visibly overcompensating. 
\begin{figure}[t]
\hspace{-2cm}
\includegraphics[width=.425\textwidth]{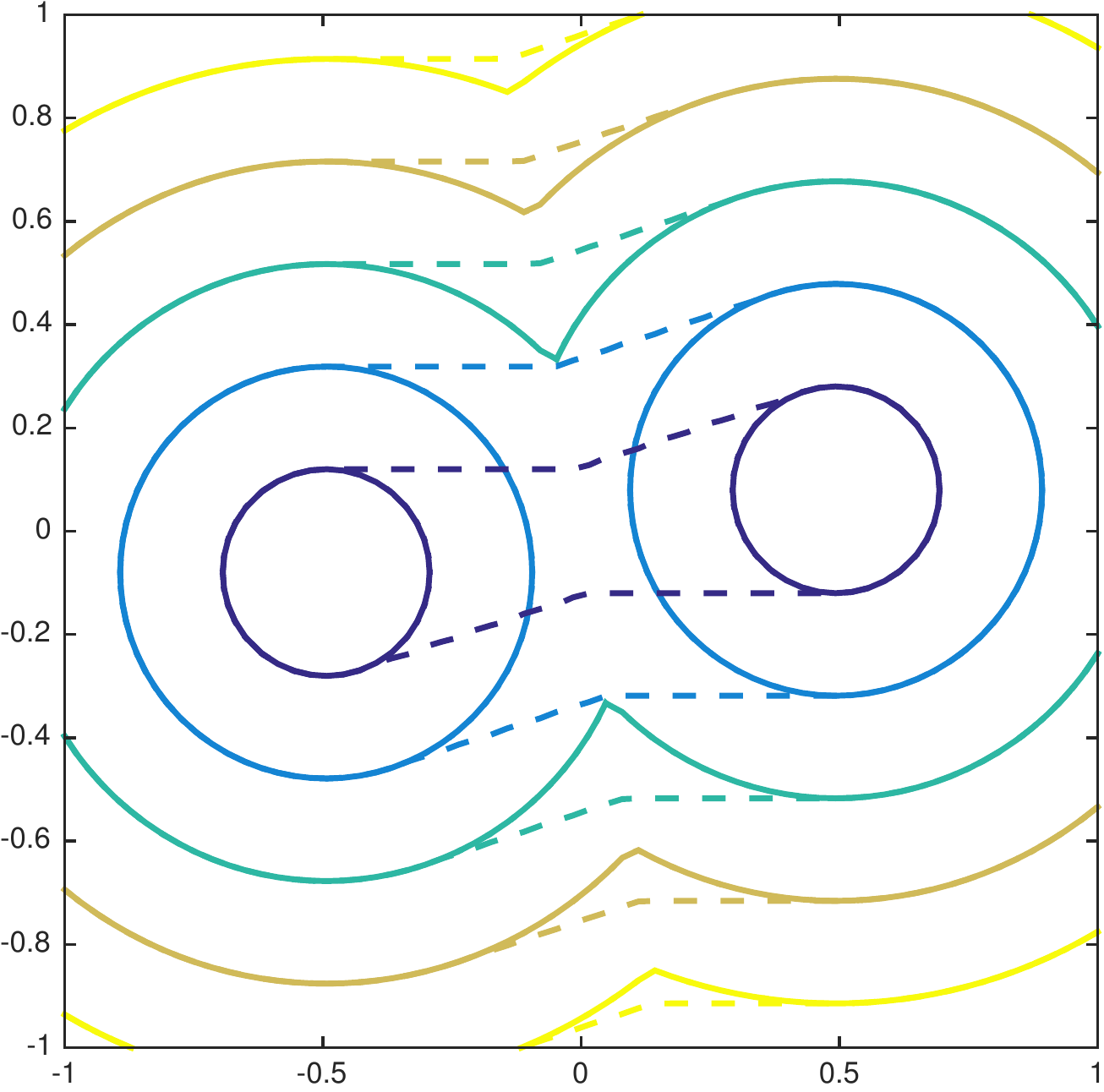}
\hspace{1.5cm}
\includegraphics[width=.425\textwidth]{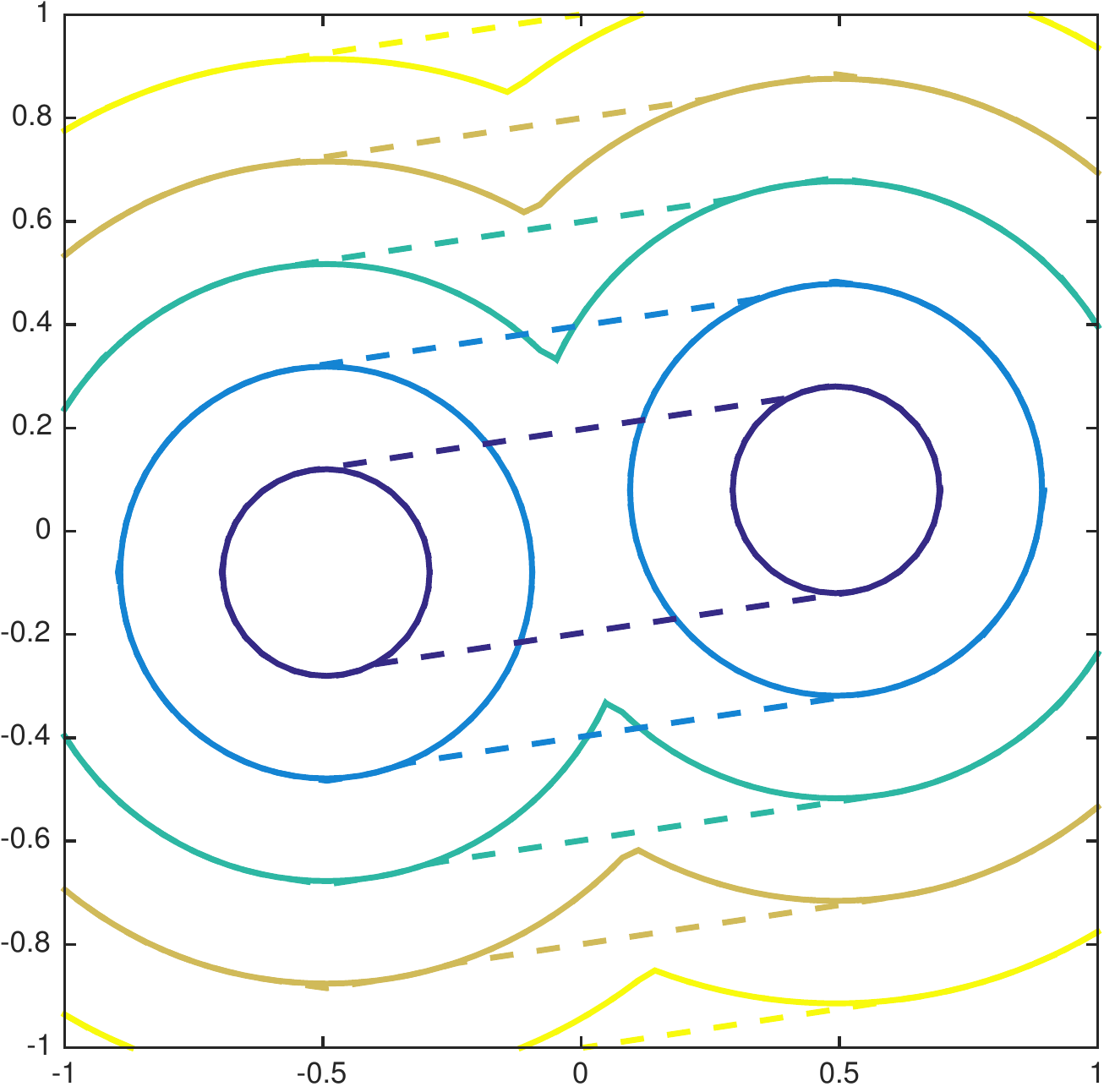}
\caption{Example \ref{example_nonaligned}.  Left: the $\Dir$-QCE can miss convexities along missing directions. Right: The $\e$-robust QCE (with $\e=0.15$) corrects this without using more direcitons}
\label{fig:nongridaligned}
\end{figure}

\end{example}

\begin{example}[Robust QCE]
\label{example_robust}
This example demonstrates how, in two dimensions, the robust QCE perturbs flat parts of the function (which are not minimizers).   First we define as continuous one dimensional function so that it 
is piecewise linear with slopes $40$ then 0, then $40$.  In particular, let $f$ interpolate the values, $f(0) = 0$, $f(.25) = 10$, $f(.75) = 10$, and extend with slope $40$.

Then take 
\[
g(x) = f(d(x))
\]
where $d(x)$ has square level sets, $d(x) = \max (\abs{x_1}, \abs{x_2})$. 
Although $g$ is quasiconvex, it is not $\e$-robustly quasiconvex; it is flat along the level set $g=10$.
We computed the  $\e=0.8$-robustly quasiconvex envelope. Results using a direction set of width $5$ are displayed in \autoref{fig:robust}.  The top images (which are inverted for visualization purposes) show a surface plot of $g$ and $u$.  The flat part is evident for $g$.  The function $u$ is bumped up (as visualized).  The lower images show contour lines, zoomed in to emphasize the large flat part for $g$ on the 10-level set.  Then level sets of $u$ spread out from 7 to 10, with a small square becoming rounder at the 8 and 9 level set, and then becoming square again by the 11 level set.   The level sets are rounder, but they are clearly not circles, they appear to have peaks along the midpoints of the flat parts of the level sets of $g$, which correspond to the axes. 
\end{example}

\begin{figure}[t]
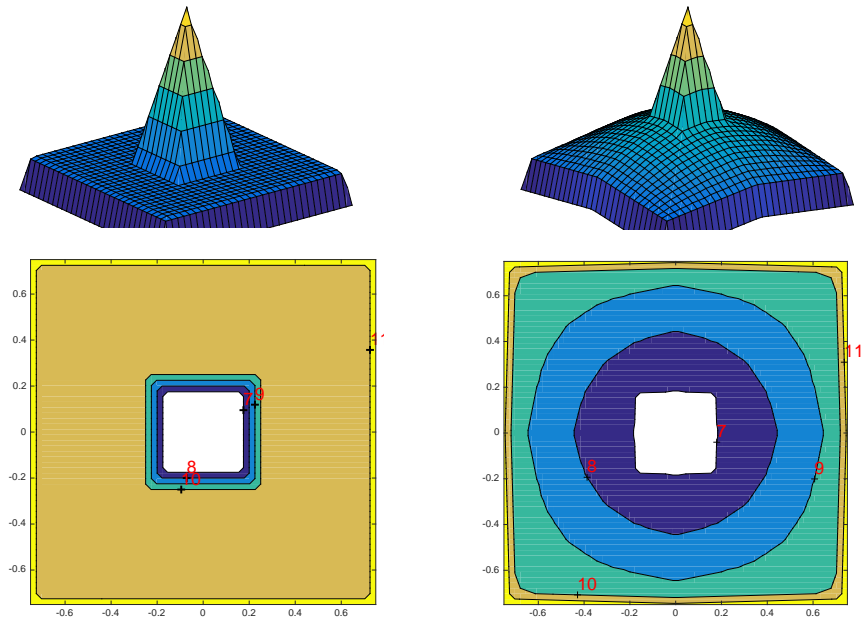

\centering
\includegraphics[width=.35\textwidth]{ls_robust_surff}
\hspace{2cm}
\includegraphics[width=.35\textwidth]{ls_robust_surfu}
\\
\includegraphics[width=.4\textwidth]{ls_robust_contourf}
\hspace{1cm}
\includegraphics[width=.4\textwidth]{ls_robust_contouru}
\caption{Example \ref{example_robust}. Top: Surface plot (inverted) of a function and its robust QCE. Bottom: Contours plot.}
\label{fig:robust}
\end{figure}

\begin{example}[3-D example]
We also consider the 3-dimensional analogue of $g^{\theta,\alpha}$, where now one of the cones has been shifted down:
\[
g(x) = \min\left\{\sqrt{(x_1-0.5)^2+x_2^2+x_3^2},\sqrt{(x_1+0.5)^2+x_2^2+x_3^2}-1/3\right\}
\]
The results, computed on a $64\times64\times64$ grid, are found in \autoref{3D_example}.

\begin{figure}[t]
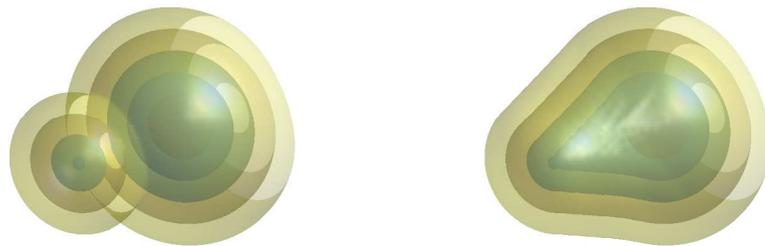

\centering
\includegraphics[width=.49\textwidth]{3d_maxcones_surf}
\includegraphics[width=.49\textwidth]{3d_maxcones_QCEsurf}
\caption{Level surfaces of $g$ (left) and its quasiconvex envelope (right). A stencil of width 4 was used on a $64\times64\times64$ grid.}
\label{3D_example}
\end{figure} 

\end{example}

\bibliographystyle{alpha}
\bibliography{QC_arxiv}

\end{document}